\documentclass[11pt,reqno]{amsart}

\usepackage{graphicx}

\newcommand{\K}{{\mathbf K}^\crit}

\newcommand{\RR}{\mathbb{R}}

\newcommand{\CC}{{\mathbb C}}

\newcommand{\C}{{\mathbb C}}

\newcommand{\N}{{\mathbb N}}

\newcommand{\dbar}{\bar\partial}
\newcommand{\ddbar}{\partial\dbar}

\newcommand{\D}{{\mathbf D}}
\newcommand{\E}{{\mathbf E}\,}

\newcommand{\half}{{\frac{1}{2}}}

\newcommand{\ecal}{\mathcal{E}}

\newcommand{\hcal}{\mathcal{H}}
\newcommand{\ical}{\mathcal{I}}

\newcommand{\mcal}{\mathcal{M}}

\newcommand{\pcal}{\mathcal{P}}

\newcommand{\ocal}{\mathcal{O}}

\newcommand{\vcal}{\mathcal{V}}

\def    \half   {{\frac{1}{2}}}

\def    \C  {{\mathbb C}}

 \def   \half   {{\frac{1}{2}}}

\newtheorem{theo}{{\sc Theorem}}[section]
\newtheorem{cor}[theo]{{\sc Corollary}}
\newtheorem{conj}[theo]{{\sc Conjecture}}

\newtheorem{remark}[theo]{{\sc Remark}}
\newtheorem{lem}[theo]{{\sc Lemma}}
\newtheorem{prop}[theo]{{\sc Proposition}}
\newtheorem{definition}[theo]{{\sc Definition}}

\newenvironment{defin-no-number}{\medskip\noindent{\it Definition:\/} }{\medskip}

\def\h#1{\hbox{#1}}

\def\o{\omega}

\def\MA{Monge-Amp\`ere }

\def\K{K\"ahler }
\def\Kno{K\"ahler}
\def\ra{\rightarrow}

\def\vp{\varphi}

\def\isom{\cong}

\def\w{\wedge}
\def\i{\sqrt{-1}}
\def\text{\textstyle}

\def\ra{\rightarrow}

\def\isom{\cong}

\def\del{\partial}

\def\KE{K\"ahler--Einstein }

\def\calP{\pcal}

\def\calM{\mcal}

\def\calH{\hcal} 

\def\calV{\vcal}
\def\calE{\ecal}

\def\calI{\ical}
\def\calO{\ocal}

\def\ginv{g^{-1}}
\def\tr{\hbox{\rm tr}}
\def\D{\Delta}
\def\Vol{\hbox{\rm Vol}}
\def\half{\hbox{$\textstyle\frac12$}}
\def\quarter{\hbox{$\textstyle\frac14$}}
\def\eigth{\hbox{$\textstyle\frac18$}}
\def\sixtienth{\hbox{$\textstyle\frac1{16}$}}

\font\sml=cmr6
\font\smlsev=cmr7

\font\smlfive=cmr5
\font\calfoot=cmsy7

\def\calVsml{{\hbox{\calfoot V}}}

\def\gE{g_{\hbox{\sml E}}}
\def\gC{g_{\hbox{\sml C}}}

\def\gEsml{g_{\hbox{\smlfive E}}}
\def\gCsml{g_{\hbox{\smlfive C}}}
\def\gV{g_{\hbox{\calVsml}}}
\def\gtildeV{g_{\widetilde{\hbox{\calVsml}}}}
\def\gM{g_{\hbox{\sml M}}}
\def\dM{d_{\hbox{\sml M}}}
\def\dE{d_{\hbox{\smlsev E}}}

\def\LE{L_{\hbox{\smlsev E}}}
\def\dC{d_{\hbox{\smlsev C}}}
\def\dV{d_{\hbox{\calVsml}}}
\def\dtildeV{d_{\tilde \calVsml}}

\def\E#1#2{(#1,#2)_{\hbox{\sml E}}}
\def\C#1#2{(#1,#2)_{\hbox{\sml C}}}

\def\Esq#1{|#1|^2_{\hbox{\sml E}}}
\def\Enorm#1{|#1|_{\hbox{\sml E}}}
\def\Vsq#1{|#1|^2_{\hbox{\calVsml}}}

\def\Ric{\hbox{\rm Ric}\,} 
\def\ovp{{\o_{\vp}}}
\def\gvpinv{g_\vp^{-1}}
\def\mathoverr#1#2{\buildrel #1 \over #2}
\def\nablaoo{{\nabla^{1,1}}}
\def\MA{\hbox{\rm MA}\,}
\def\polishl{\char'40l}
\def\Kolodziej{Ko\polishl{}odziej}
\def\Blocki{B\polishl{}ocki}

\providecommand{\tilv}{{\widetilde{\mathcal{V}}}}
\providecommand{\norm}[1]{\lVert#1\rVert}

\title[Ricci flow and the completion of the space of K\"ahler metrics
]
{
Ricci flow and the metric completion of the space of K\"ahler metrics
}

\author{Brian Clarke }
\author{Yanir A. Rubinstein }

\address{Department of Mathematics, Stanford University, Stanford, CA 94305, USA}
\email{bfclarke@stanford.edu, yanir@member.ams.org}

\thanks{\hglue-10pt February 17, 2011.}

\begin{document}

\begin{abstract}

We consider the space 
of K\"ahler metrics as a Riemannian submanifold of the space 
of Riemannian metrics, and study the associated submanifold geometry. 
In particular, we show that the intrinsic and extrinsic distance functions
are equivalent. We also determine the metric completion of the space
of \K metrics, making contact with recent generalizations of the Calabi--Yau
Theorem due to Dinew, Guedj--Zeriahi, and \Kolodziej. 
As an application, we 
obtain a new analytic stability criterion for the existence of a
K\"ahler--Einstein metric on a Fano manifold in terms of the Ricci
flow and the distance function. We also prove that the K\"ahler--Ricci
flow converges as soon as it converges in the metric sense.

\end{abstract}

\maketitle

\section{Introduction}

The study of the infinite-dimensional space $\calH$ of all
\K metrics in a fixed \K class has evolved
essentially independently of the study of the larger space $\calM$ of all
Riemannian metrics on a closed, finite-dimensional base manifold $M$.  
Our first purpose in this article is to draw attention to a simple connection
between the two, going back to Calabi, which does not seem to be
well known. Namely, we consider the space of \K metrics as a
submanifold of the space of Riemannian metrics, and study the induced
intrinsic and extrinsic geometry of 
$\calH$ under this embedding.  

Our first main result is that when the space
of all Riemannian metrics is equipped with 
the Ebin metric (often referred to as the $L^2$ metric), the
intrinsic and extrinsic distance functions are equivalent.  At the
same time, the subspace of \K metrics is in a sense as far from being
totally geodesic as possible---in fact, it shares no common geodesics
with the ambient space, and geodesics in the ambient space intersect
the subspace 
in at most two points.

Building on the equivalence result, we then determine the (metric) 
completion of $\calH$, making contact with some recent deep
results in pluripotential theory, due to Dinew, Guedj--Zeriahi and \Kolodziej,
that generalize the Calabi--Yau Theorem.

These results, combined with recent deep results on
the Ricci flow, are then used to prove a new
analytic characterization of K\"ahler--Einstein manifolds
of positive scalar curvature in terms of the Ricci flow
and the induced distance function. 
This result stands in clear analogy with Donaldson's conjecture
regarding ``geodesic stability,"
with Ricci flow paths taking the place of geodesic rays.
It follows that for the K\"ahler--Ricci flow,
convergence in the induced metric implies smooth convergence.
This, and a related analytic condition that is also
shown to be equivalent to smooth convergence, strengthen
some recent results due to 
Phong--Song--Sturm--Weinkove.

We note that the study of ``constrained" distance and geodesics also
appears naturally in optimal transport in relation to the Wasserstein
metric.  In fact, Carlen--Gangbo \cite{CG} consider a submanifold of
the space of probability measures and study its induced geometry.
There are a number of analogies between their approach, as well as
their results, and the ones in this article.  For instance, the
submanifold they study is naturally a hypersurface---a portion of a
sphere---and a similar situation appears for the space of \K metrics.

Let $M$ be a smooth, closed (i.e., compact and without boundary)
manifold, and consider the infinite-dimensional space $\calM$ of all
smooth Riemannian metrics on $M$. The space $\calM$ may be endowed
with a natural Riemannian structure, which we refer to as the
Ebin metric \cite{E} (cf. \cite{D}), defined as follows,
\begin{equation}
\gE(h,k)|_g:=\int_M \tr (\ginv h \ginv k) dV_g,
\end{equation}
where $g\in\calM$, $h,k\in T_g\calM$ and
$T_g\calM\isom \Gamma(\h{\rm Sym}^2T^\star\! M)$, 
the space of smooth, symmetric $(0,2)$-tensor fields on $M$.
As shown by Freed--Groisser \cite{FG} and Gil-Medrano--Michor \cite{GM},
the curvature of $\gE$ is nonpositive and geodesics
satisfy the equation
\begin{equation}
\label{EbinGeodEq}
(\ginv g_t)_t=
\frac14\tr(\ginv g_t\ginv g_t)\delta
-\frac12\tr(\ginv g_t)\ginv g_t,
\end{equation}
where $\delta$ denotes the Kronecker tensor. The geodesics can be
computed explicitly, however the metric is incomplete, and
in general not every two points can be connected by a geodesic. 
Nevertheless, it has been shown recently that $(\calM,\gE)$ is 
a metric space \cite{Cl1,Cl2}, and a detailed description of
its completion has been provided, including an explicit 
computation of the length-minimizing paths in it
and its distance function $\dE$ \cite{Cl5}.

Now, assume that $M$ admits a \K structure $(M,J,\omega)$, and let
$\calH\subset \calM$ denote the space of all smooth \K metrics on
$(M,J)$ whose \K form is cohomologous to $\omega$. Let $n$ denote
the complex dimension of $M$ and $V$ denote the total volume $M$ with
respect to $\omega$ (which depends only on the cohomology class of
$\omega$).  The space $\calH$, by the $\ddbar$-lemma \cite{GH}, may be
parametrized by a single smooth function, the \K potential,
$$
\calH:=\{g_\vp\,:\, \o_\vp:=\o+\i\ddbar\vp>0\}\subset\calM,
$$
where $g_\vp(\,\cdot\,,\,\cdot\,):=\omega_\vp(\,\cdot\,,J\,\cdot\,)$, and
$\varphi$ is unique up to an additive constant. The corresponding
space of \K potentials is denoted by
$$
\calH_\o:=\{\vp\,:\, \o_\vp:=\o+\i\ddbar\vp>0\}\subset C^\infty(M),
$$
and $\calH\isom\calH_\o/\RR$.
There are several natural candidates for metrics on $\calH$.  
The most widely studied is the Mabuchi metric \cite{M},
\begin{equation}
\label{MabuchiMetricEq}
\gM(\nu,\eta)|_\vp:=\int_M\nu\eta\,\o_\vp^n,
\quad \nu,\eta\in T_\vp\calH_\o\isom C^\infty(M),
\end{equation}
discovered independently also by Semmes \cite{S} and Donaldson
\cite{Do1} (see, e.g., \cite{Ch1} or \cite[Chapter 2]{R1} for an 
exposition and further references). Calabi and Chen proved that
$\gM$ induces a metric space structure on $\mathcal{H}$, 
and that this space has nonpositive curvature in the sense of Alexandrov
\cite{Ch1,CCh}.

Similarly, one may consider 
metrics involving more (or fewer) derivatives.  The Calabi metric 
is defined by
\begin{equation}
\gC(\nu,\eta)|_\vp:=\int_M\D_\vp\nu\D_\vp\eta\,\frac{\o_\vp^n}{n!}.
\end{equation}
This metric was introduced by Calabi in the 1950s in talks
and in a research announcement \cite{C1,C2}, however, since Calabi's
construction depends on---and in fact seems to have prompted---the Calabi--Yau Theorem (see Remark 
\ref{CYThmRemark}), 
the detailed computations leading to his results have appeared in print only 
in a recent article of Calamai \cite{Ca}. In this metric, $\calH$
is a section of a sphere, (i.e., has constant positive sectional curvature) 
of finite diameter, and any two points can be connected by a unique 
(explicit) smooth minimizing geodesic.

The article is organized as follows.
Our first, and elementary, observation, which is undoubtly due to Calabi, 
is that the metric $\gC$
on $\calH$ is simply the metric induced by $\gE$ under the inclusion
$\calH\hookrightarrow\calM$ (Proposition \ref{EbinCalabiProp}).
Thus, as in the situation studied by Carlen--Gangbo, our submanifold
is a portion of a sphere.
In \S\ref{IISection}, the second fundamental form of the inclusion
$\iota_\calH:\calH\hookrightarrow(\calM,\gE)$ is computed,
relying on results of
Ebin, Freed--Groisser, and Gil-Medrano--Michor on the geometry of
$(\calM,\gE)$.
It follows that no geodesic in the Calabi metric is a geodesic in the
Ebin metric, and that geodesics of the Ebin metric 
intersect the space of \K metrics 
in at most two points (Remark \ref{GeodIntersect}).
Then, we prove that the extrinsic and intrinsic distance functions
$\dE$ and $\dC$, respectively, are equivalent on the space of \K
metrics (Theorem \ref{EquivGeneral}).  To do so, we use a
transformation of the ambient space that makes the spherical nature of
$\calH$ self-evident and---analogously to \cite{CG}---compare intrinsic
geodesics (great circles) to extrinsic geodesics (chords).
Motivated by the proof of the equivalence result, we then formulate a criterion
for $\dC$-convergence which improves the criterion for 
$\dE$-convergence \cite[Thm.~4.15]{Cl4} in the ambient space
(\S\ref{CompletionSection}). Next,
we determine the completion of $\calH$ with respect to the geometry
induced by $(\calM,\gE)$ (Theorem \ref{calHdCCompletion}), building upon Theorem \ref{EquivGeneral}
and recent generalizations of the Calabi--Yau Theorem.
In \S\ref{RicciFlowSubsec} we define the notion of Calabi--Ricci stability and
prove, building on the description
of the completion of $\calH$, that it is equivalent to the 
existence of a \KE metric on a Fano manifold (Theorem \ref{CRStabilityThm}).
It follows that the K\"ahler--Ricci flow converges smoothly as soon as 
it $\dC$-converges (Corollary \ref{dCConvSmoothConvCor}), strengthening
a theorem of Phong et al.~\cite{PSSW}.
We also obtain an improved analytic characterization of convergence of the 
flow (Corollary \ref{LoneLtwoConvCor}).
We conclude with some remarks and directions for future study
in \S\ref{FutureStudySection}.

\section{The induced metric}
\label{InducedMetricSection}

We begin by considering the restriction of the Ebin
metric $\gE$ to the space of \K metrics $\calH$.
The computations involve 
the K\"ahler--Riemannian dictionary of translating Hermitian objects 
written with respect to holomorphic coordinates to their
Riemannian counterparts written in real coordinates.
However, we include the detailed, completely elementary, computations
in the proofs in this section since the exact constants are 
important for us later, and in order to avoid confusion between
different possible conventions.

Given a \K metric $g\in\calH\subset\calM$ and a holomorphic
coordinate chart $z_1,\ldots,z_n$, denote by $[g_{i\bar \jmath}]$
the corresponding Hermitian matrix, 
$g_{i\bar \jmath}=g(\frac\del{\del z^i},\frac\del{\del\overline{z^j}})$.
Denote by $[g_{ij}]$ the matrix of coefficients of the metric $g$ regarded
as a Riemannian metric, i.e., $g_{ij}=g(\frac\del{\del x_i},\frac\del{\del x_j}), 
i,j\in\{1,\ldots,2n\}$, with
respect to the real coordinates $x_1,\ldots,x_{2n}$, where
$z_i=x_i+\i x_{i+n},$ $i=1,\ldots,n$. If $G=[g_{i\bar \jmath}]$ then
$g_{i\bar \jmath}=\frac12 g_{ij}+\frac\i2 g_{ij+n}$.
In matrix notation,
\begin{equation}
\label{GtoGEq}
[g_{ij}]
=
\begin{pmatrix}
G+\bar G&(G-G^T)/\i\cr (G^T-G)/\i&G+\bar G
\end{pmatrix}.
\end{equation}

For a function $f\in C^\infty(M)$, we denote
$$
[(\nabla^2 f)_{ij}]
=
\Big[\frac{\del^2 f}{\del x^i\del x^j}\Big]
=
\begin{pmatrix}
A_f&B_f^T\cr B_f&C_f
\end{pmatrix},
$$
then the complex Hessian is given by
\begin{equation}
\label{CxHessianEq}
[f_{i\bar \jmath}]
=
\bigg[\frac{\del^2 f}{\del z^i\del\overline{z^j}}\bigg]
=\frac14(A_f+C_f)+\frac\i4(B_f^T-B_f).
\end{equation}
We consider the map $\iota_{\calH_\o}:\calH_\o\ra\calM$, given as
the composition
\begin{equation}
\label{HoInclusionEq}
\calH_\o
\xrightarrow{g_{\o+\i\ddbar(\mskip3mu\cdot\mskip3mu)}}
\calH\mathoverr{\iota_\calH}\hookrightarrow\calM.
\end{equation}
Its differential
$d\iota_{\calH_\o}:C^\infty(M)\ra\h{\rm Sym}^2T^\star\! M$
is independent of the point $\vp\in\calH_\o$.
By combining (\ref{GtoGEq}) and (\ref{CxHessianEq}), we see
that in local coordinates, $d\iota_{\calH_\o}f$ 
is given by
\begin{equation}
\label{DIotaHEq}
\frac12
\begin{pmatrix}
A_f+C_f&B_f^T-B_f\cr B_f-B_f^T&A_f+C_f
\end{pmatrix}.
\end{equation}
In fact,
\begin{equation}
\label
{NablaOneOneEq}
d\iota_{\calH_\o}=P^{1,1}\circ\nabla^2=:\nabla^{1,1},
\end{equation}
where $P^{1,1}$ denotes the projection of a symmetric $(0,2)$-type
tensor onto its $J$-invariant part, and the action of $J$ on $\h{\rm
  Sym}^2T^\star\! M$ is given by $J \cdot h:=h(J\,\cdot,J\,\cdot\,)$.
The formula \eqref{NablaOneOneEq} holds since $J \cdot \nabla^2f$ is
represented in coordinates by
$$
\begin{pmatrix}
0&-I\cr I&0
\end{pmatrix}
\begin{pmatrix}
A_f&B_f^T\cr B_f&C_f
\end{pmatrix}
\begin{pmatrix}
0&I\cr -I&0
\end{pmatrix},
$$
and $P^{1,1}h=\half(h+J \cdot h)$. 
In this notation, if we let $g_\vp$ denote the metric associated to $\o_\vp$,
then $g_{\vp + f}=g_\vp+\nablaoo f$.

We note that from this description of metrics in $\mathcal{H}$, we
see that $\mathcal{H}$ is the intersection of a closed
affine subspace (within the space of symmetric $(0,2)$-tensor fields)
with $\mathcal{M}$.  Indeed, if $g_0$ is the Riemannian metric
associated to the reference \K form $\omega$, then by the above
discussion any metric $g_\vp \in \mathcal{H}$ is given by $g_0
+ \nablaoo \varphi$.  This shows, in particular, that $\mathcal{H}$ is
an embedded submanifold of $\mathcal{M}$.

Using these preliminaries, we make the following observation, which we
believe is due to Calabi.  It does not seem to be well-known, and
serves as our starting point. It shows that $(\calH,2\gC)$ is
isometrically embedded in $(\calM,\gE)$.

\begin{prop}
\label{EbinCalabiProp}
Consider the inclusion $\iota_{\calH}:\calH\hookrightarrow\calM$.
Then, $\iota_{\calH}^\star\, \gE = 2\gC$.
\end{prop}

Here and in the sequel, we abuse notation by using $\gC$ to also
denote the metric on $\calH$ obtained by pushing $\gC$ forward to
$\calH$ under the first map in (\ref{HoInclusionEq}) (i.e., we also
write $\iota_{\calH_\o}^\star\, \gE = 2\gC$).  The $\ddbar$-lemma
implies that there is no loss in doing so.
We also note that by abuse of notation, we often write both $\o\in\calH$
and $g_\o\in\calH$.

\begin{proof}
First, note that $\gC$ may be alternatively expressed as
\begin{equation}
\label{AlternativeGCEq}
\gC(\nu,\eta)|_\vp=\int_M(\i\ddbar\nu,\i\ddbar\eta)_{\o_\vp}\,\frac{\o_\vp^n}{n!}.
\end{equation}
To see this, recall the following algebraic identity for any 
$(1,1)$-forms $\beta,\gamma$ and a strictly 
positive $(1,1)$-form $\alpha$ 
(\cite{A},\cite[Lemma 2.77]{B}),
\begin{equation}
\label{AlgebraicIdentityEq}
(\alpha,\beta)_{\alpha}(\alpha,\gamma)_{\alpha}
-
(\beta,\gamma)_{\alpha}
=
\frac{\beta\w\gamma\w\alpha^{n-2}/(n-2)!}
{\alpha^n/n!}.
\end{equation}
Since the right-hand side in this identity is exact whenever
$\beta$ and $\gamma$ are,
and since $(\o_\vp,\i\ddbar\nu)_{\o_\vp}=\D_\vp\nu$, equation (\ref{AlternativeGCEq}) follows.

We claim that 
\begin{equation}
\label{AlternativeInnerProductEq}
2(\i\ddbar\nu,\i\ddbar\eta)_{\o_\vp}=\tr(\gvpinv\,d\iota_{\calH_\o}\nu\, \gvpinv\,d\iota_{\calH_\o}\eta),
\end{equation}
where $d\iota_{\calH_\o}\nu=\frac{d}{dt}\big|_{t=0}g_{\vp+t\nu}$ is as in (\ref{NablaOneOneEq}),
and similarly for $\eta$.
For the proof, it is enough
to verify this identity pointwise. 
If $\o_\vp=\i g_{i\bar \jmath}dz^i\w
d\overline{z^j}$, then the left-hand side of
(\ref{AlternativeInnerProductEq}) is
$$
g_\vp^{i\bar l}g_\vp^{k\bar \jmath}\nu_{i\bar \jmath}\eta_{k\bar l}
=\tr(G_\vp^{-1}RG_\vp^{-1}S),
$$
where $G_\vp=[(g_\vp)_{i\bar \jmath}], R=[\nu_{i\bar \jmath}], S=[\eta_{i\bar \jmath}]$.
By choosing holomorphic normal coordinates at $p\in M$,
$(g_\vp)_{i\bar \jmath}(p)=\delta_{i\bar \jmath}$, and by
(\ref{GtoGEq}) we have
$(g_\vp)_{ij}(p)=2\delta_{ij}$. 
The left-hand side of (\ref{AlternativeInnerProductEq}) equals
$$
\begin{aligned}
\sixtienth\tr\big((A_\nu+C_\nu+\i B_\nu^T-\i B_\nu)(A_\eta+C_\eta+\i B_\eta^T-\i B_\eta)\big),
\end{aligned}
$$
while the right-hand side equals
$$
\tr
\bigg(
2^{-1}\frac12
\begin{pmatrix}
A_\nu+C_\nu&B_\nu^T-B_\nu\cr B_\nu-B_\nu^T&A_\nu+C_\nu
\end{pmatrix}
2^{-1}\frac12
\begin{pmatrix}
A_\eta+C_\eta&B_\eta^T-B_\eta\cr B_\eta-B_\eta^T&A_\eta+C_\eta
\end{pmatrix}
\bigg),
$$
proving (\ref{AlternativeInnerProductEq}).

To conclude the proof, observe that 
$\o^n/n!=\det[g_{i\bar \jmath}]\bigwedge_{k=1}^n\i dz^k\w d\bar{z^k}$,
while $dV_g=\sqrt{\det[g_{ij}]}\bigwedge_{k=1}^{2n} dx^k$.
Note that if $G=A+\i B$ then $[g_{ij}]
=2\begin{pmatrix}A&B\cr-B&A\end{pmatrix}$,
hence $\det[g_{ij}]=2^{2n}|\det [g_{i\bar \jmath}]|^2$ (see, e.g., \cite[Lemma 2]{CP}).
Hence $dV_g=\o^n/n!$, and the proposition follows. 
\end{proof}

Since, as we recall in \S\ref{GeodesicsCalabiSubsec}, $(\calH,\gC)$
has diameter $\pi\sqrt{V}$, it follows that $\calH$ is a bounded 
set in $(\calM,\gE)$ (of diameter at most $\pi\sqrt{2V}$). 
This also follows directly from the fact
the set of all metrics of volume not greater than $V$ in $\calM$
has diameter at most $4 \sqrt{\frac{2V}{n}}$ \cite[Prop.~4.1]{Cl3}.
The latter is a better bound whenever $n>1$, reflecting to some
degree the extent to which $\calH$ is far from being totally geodesic,
as we show in the next subsection.

As a corollary of the proof of Proposition \ref{EbinCalabiProp}, 
we record the following property of tangent vectors to $\calH\subset \calM$.

\begin{lem}
\label{TanHTrace}
  For all $h, k \in T_{g_\varphi} \mathcal{H}$, 
  \begin{equation*}
    \E hk = \frac{1}{4n} \E{\tr(g_\varphi^{-1} h) g_\varphi}
    {\tr(g_\varphi^{-1} k) g_\varphi} = \frac{1}{2} \int_M
    \tr(g_\varphi^{-1} h) \tr(g_\varphi^{-1} k) \, dV_{g_\varphi}. 
  \end{equation*}
\end{lem}
\begin{proof}
Let $\nu,\eta \in T_\varphi \mathcal{H}_\omega$, 
and let $\nabla^{1,1}\nu,\nabla^{1,1}\eta\in T_{g_\vp}\calH\subset T_{g_\vp}\calM$
(recall (\ref{NablaOneOneEq})). 
Then by (\ref{CxHessianEq}) and (\ref{DIotaHEq}),
\begin{equation}
\label{LaplacianRealCxEq}
\Delta_\varphi \nu
= 
\frac{1}{2} \tr(g_\varphi^{-1} \nabla^{1,1}\nu)
\end{equation}
(remembering that if 
$g_{i\bar \jmath}(p)=\delta_{i\bar \jmath}$ then $g_{ij}(p)=2\delta_{ij}$).
Now, by the $\ddbar$-lemma, $\nabla^{1,1}$ is an isomorphism
between $C^\infty(M)/\RR$ and $T_{g_\vp}\calH$. Hence,
given $h,k\in T_{g_\varphi} \mathcal{H}$ there
exist $\nu,\eta\in C^\infty(M)$ with $h=\nabla^{1,1}\nu,\,
k=\nabla^{1,1}\eta$.
So by Proposition \ref{EbinCalabiProp} and (\ref{LaplacianRealCxEq}),
  \begin{equation*}
    \E hk = 2\C \nu\eta= 2 \int_M \Delta_\varphi \nu
    \Delta_\varphi \eta \, \frac{\omega_\varphi^n}{n!} = \frac{1}{2}
    \int_M \tr(g_\varphi^{-1} h) \tr(g_\varphi^{-1} k) dV_{g_\varphi},
  \end{equation*}
as claimed.
\end{proof}

\begin{remark}
{\rm
Lemma \ref{TanHTrace} may be interpreted as saying that the angle 
cut out between $\calH$ and the conformal classes is a constant
depending only on the dimension. For more on this
we refer to \S\ref{sec:angl-betw-mathc}.

}
\end{remark}

\section{The second fundamental form}
\label{IISection}

We now compute the second fundamental form $II$ of 
$\iota_\calH:\calH\hookrightarrow\calM$.
For simplicity we state the result only in terms of 
the trace of $II$. There is no loss in doing so since
$II$ may be recovered from its trace by using
the Green's operator. 
It then follows that 
no geodesic of $(\calH,\gC)$ is a geodesic of $(\calM,\gE)$.

\begin{prop}
\label{IIProp}
The trace of the second fundamental form of the inclusion 
$\iota_\calH:\calH\hookrightarrow (\calM,\gE)$ is given by
$$
\tr(g_{\vp}^{-1} II(h,k))\big|_{g_\vp}
=
-\frac n2\tr(\gvpinv h\gvpinv k)
+\frac14\tr(\gvpinv h)\tr(\gvpinv k)
-\frac1{2V}\E hk,
$$
where $h=\nabla^{1,1}\nu,k=\nabla^{1,1}\eta$, with $\nu,\eta\in C^\infty(M)$ constant
vector fields on $\calH$, and $\nabla^{1,1}$ defined by (\ref{NablaOneOneEq}).
In particular, no geodesic in $(\calH,\gC)$ is a geodesic in $(\calM,\gE)$.

\end{prop}

\begin{proof}

For the following formula we refer to 
\cite[p. 19]{E}, or \cite[p. 335]{FG} (or \cite[p. 189]{GM}
with a different sign convention).

\begin{lem}
The Levi-Civita connection of $(\calM,\gE)$ is given by
$$
\nabla^{\gEsml}_hk\big|_g=
-\frac12h\ginv k-\frac12k\ginv h
-\frac14\tr(\ginv h\ginv k)g+\frac14\tr(\ginv h)k+\frac14\tr(\ginv k)h,
$$
for constant vector fields $h,k\in T\calM$.
\end{lem}

Next, we compute the Levi-Civita connection of $(\calH_\o,\gC)$.
We first compute this on the level of \K potentials, and then
translate to the level of metrics.
Let $\nu,\eta,\psi\in T_\vp\calH_\o\isom C^\infty(M)$
be constant vector fields.
Since
$$
\frac d{ds}\Big|_{s=0}\Delta_{\vp+s\psi}\nu=-(\i\ddbar\psi,\i\ddbar\nu)_\ovp,
$$
it follows that 
\begin{multline*}
  \psi(\nu,\eta)=  \int_M
\big(
\Delta_\vp\nu\D_\vp\eta\Delta_\vp\psi 
-(\i\ddbar\psi,\i\ddbar\nu)_\ovp\Delta_\vp\eta \\
-(\i\ddbar\psi,\i\ddbar\eta)_\ovp\Delta_\vp\nu
\big)\frac{\ovp^n}{n!}.
\end{multline*}
The Koszul formula then gives that $\Delta_\vp(\nabla^{\gCsml}_\nu\eta|_\vp)$
is equal, up to a constant, to
$$
-(\i\ddbar\eta,\i\ddbar\nu)_{\o_\vp}+\half\Delta_\vp\eta\Delta_\vp\nu.
$$
By (\ref{AlgebraicIdentityEq}) it follows that
$$
\Delta_\vp(\nabla^{\gCsml}_\nu\eta|_\vp)
=
\half\Delta_\vp\eta\Delta_\vp\nu+\half V^{-1}\int_M(\Delta_\vp\eta\Delta_\vp\nu)\frac{\ovp^n}{n!}
-(\i\ddbar\eta,\i\ddbar\nu)_{\o_\vp}.
$$
On the level of \K forms the tangent vector is expressed as
\begin{equation}
\label{CalabiLCConnEq}
\i\ddbar\Delta_\vp^{-1}\left(
\half\Delta_\vp\eta\Delta_\vp\nu+\half V^{-1}\int_M(\Delta_\vp\eta\Delta_\vp\nu)\frac{\ovp^n}{n!}
-(\i\ddbar\eta,\i\ddbar\nu)_{\o_\vp}\right).
\end{equation}
The corresponding
tangent vector in $T_{g_\vp}\calM$ is given by 
$\nabla^{1,1}(\nabla^{\gCsml}_\nu\eta|_\vp)$ (recall (\ref{NablaOneOneEq})).
Let $h=\nabla^{1,1}\nu$ and $k=\nabla^{1,1}\eta$ be elements of $T_g\calM$.
Slightly abusing notation, we have 
$\nabla^{\gCsml}_hk|_{g_\vp}=\nabla^{1,1}(\nabla^{\gCsml}_\nu\eta|_\vp)$.
By Proposition \ref{EbinCalabiProp}, (\ref{AlternativeInnerProductEq}), 
and (\ref{LaplacianRealCxEq}), on the level of metrics then, 
$$
\half
\tr(g_{\vp}^{-1}\nabla^{\gCsml}_hk)
=
\eigth\tr(\gvpinv h)\tr(\gvpinv k)
+
\quarter V^{-1}\E hk
-\half\tr(\gvpinv h\gvpinv k).
$$
Since, by Proposition \ref{EbinCalabiProp},
$$
II(h,k)=\nabla^{\gEsml}_hk-\nabla^{\gCsml}_hk,
$$
the claimed formula follows (note $\tr(g_\varphi^{-1}g_\varphi)=2n$).

Now, Lemma \ref{TanHTrace} implies
that whenever $h,k\in T_{g_\vp}\calH$,
$$
\E {II(h,k)}{g_\vp}
=
\int_M\tr(\gvpinv II(h,k))dV_{g_\vp}
=
-\frac {n}2\E hk.
$$
Hence, for any nonzero vector $h\in T_{g_\vp}\calH$ 
we have $II(h,h)\not\equiv0$, and this completes the proof
of the Proposition.
\end{proof}

Alternatively, the last conclusion
may be proved by examining the explicit expressions for
the exponential map of $\gE$.
In fact, we will see below (Remark \ref{GeodIntersect}) that 
geodesics of $\mathcal{M}$ intersect $\mathcal{H}$ in at most 
two points.

\section{Intrinsic and extrinsic distance on the space of \K metrics }

By Proposition \ref{EbinCalabiProp}, when $\calH$ is
considered as a submanifold of $(\calM,\gE)$, 
its induced metric precisely coincides with twice
the Calabi metric. Comparing the distance between
\K metrics measured with respect to these two $L^2$
metrics then corresponds to comparing the extrinsic
(Ebin) distance, and the intrinsic (Calabi)
distance. 
Let $\dE,\dC$ denote the distance functions of
$(\calM,\gE)$ and $(\calH,\gC)$, respectively.
The main result of this section (Theorem \ref{EquivGeneral})
shows that these two distance functions are equivalent.

The proof of this fact uses the Calabi--Yau Theorem and the associated
diffeomorphism $\mathcal{H} \cong \mathcal{V}$, where $\mathcal{V}$ is
the space of all smooth volume forms on $M$ with total volume $V =
\Vol(M, \omega)$.  On $\mathcal{V}$, an explicit expression for $\dC$
can be obtained, and we use this to show that $\dC$ is equivalent to
the metric induced from the Ebin metric on the ambient space
$\tilv$ of all smooth volume forms on $M$.  We then translate this
result back to $\calM$,
by using the natural submersion $\calM\ra\tilv$, and the product
structure it induces on $\calM$.

\subsection{The space of volume forms as a submanifold of $\calM$}
\label{VolFormSubmanifSubS}
Our references for this subsection are Ebin \cite{E} and 
Freed--Groisser \cite{FG} (see also \cite[\S2.5.3]{Cl1}).

Consider the space $\tilde\calV$ of all smooth volume forms on $M$.
At any point $\mu$, the tangent space to $\tilv$ is canonically
isomorphic to $\Omega^{2n}(M)$, the space of smooth $2n$-forms on $M$.
On the other hand, for each fixed $\mu\in\tilv$, consider the smooth
submanifold \cite[Lemma 8.8]{E}
$$
\mathcal{M}_\mu:= 
\{g \in \mathcal{M} : dV_g = \mu \}\subset\calM.
$$
Since the map $i_\mu:\calM_\mu\times\tilv\ra\calM$,
$(g,\nu)\mapsto (\nu/\mu)^{2/n}g$ (sending $(g,\nu)$ to the unique metric
conformal to $g$ with volume form $\nu$) is a diffeomorphism,
the space $\calM$ inherits the structure of a product manifold.
Define $\pi : \mathcal{M} \rightarrow \tilv$ by 
$\pi(g) := dV_g$. It is surjective and its differential is
$$
d \pi|_g h = \frac{1}{2} \tr(g^{-1} h) dV_g, \quad g\in\calM,\quad h\in T_g\calM.
$$
When $\calM$ is equipped with the metric $\gE$ and $\tilv$ 
with the metric $\frac4{2n} \gtildeV$, where 
\begin{equation*}
  \gtildeV(\alpha,\beta) := \int_M
    \frac{\alpha}{\mu}
    \frac{\beta}{\mu}
\mu, \quad \alpha,\beta\in T_\mu\tilv\isom\Omega^{2n}(M),
\end{equation*}
the map $\pi$ becomes
a Riemannian submersion whose vertical fibers are of
the form $\calM_\mu$, with vertical tangent spaces
$T_g^v \mathcal{M}
= \{ h \in T_g \mathcal{M} : \tr(g^{-1} h) = 0\}$ and horizontal
tangent spaces $T_g^h \mathcal{M} = C^\infty(M) \cdot g = \{h \in T_g
\mathcal{M} : h = \frac{1}{2n} \tr(g^{-1} h) g \}$.
The leaf of the horizontal distribution through $g\in\calM$
is precisely the conformal class $\calP g$ of $g$, where
$\calP:=\{\,F\in C^\infty(M)\,:\, F>0\}$, and $\calP\isom\tilv$.

\subsection{The space of fixed-volume volume forms}
\label{FixedVolumeFormsSubsec}
Consider the submanifold
$$
\calV:=\left\{\,\mu \hbox{ is a smooth volume form on $M$}\,:\, \int_M\mu=V\right\}
$$
of $\tilv$, and the map $\calH\ni g_\omega\mathoverr{\iota_{\calH,\calV}}\mapsto dV_g=\o^n/n!\in\calV$.
By the maximum principle (and the $\ddbar$-lemma), it is injective \cite{C3}.
Let $\gV$ denote the metric on $\calV$ induced from the inclusion
$\calV\hookrightarrow(\tilv,\gtildeV)$,
$$
\gV(\alpha,\beta)|_\mu=\int_M \frac\alpha\mu\frac\beta\mu\mu, 
\quad \alpha,\beta\in T_\mu\calV\isom \Omega^{2n}_0(M),
$$
where $\Omega^{2n}_0(M)$ denotes the space of smooth $2n$-forms
on $M$ that integrate to zero.
Since $d\iota_{\calH,\calV}|_{g_\o}\nablaoo\nu=\Delta_\o\nu\,\o^n/n!$,
it follows that $\iota_{\calH,\calV}^\star\gV=\gC$.
The Calabi--Yau theorem \cite{Y} states that $\iota_{\calH,\calV}$
is also surjective, and hence that
$(\calH,\gC)$ is isometric to $(\calV,\gV)$.

To summarize, we saw that: (i) when $\tilv$ is considered as a
subspace of $\calM$ via the map $i_\mu(g, \cdot)$ (with some arbitrary
choices of $g$ and $\mu$), its natural $L^2$ metric $\gtildeV$ coincides
with the one induced from $(\calM,\frac n2\gE)$; (ii) by the
Calabi--Yau theorem, Calabi's metric $\gC$ on $\calH$ induces a metric
on $\calV$; and (iii) the latter metric coincides with the
metric induced on $\calV$ from the inclusion
$\calV\hookrightarrow(\tilde\calV,\gtildeV)$.

Hence, $\calV$ inherits two distance functions---the intrinsic distance
$\dV$ from $(\calV,\gV)$ and the extrinsic distance $\dtildeV$ induced
from $(\tilv,\gtildeV)$.
In the next subsection we compute $\dV$,
and in \S\ref{IntrinsicExtrinsicVolSubS} we prove it is equivalent to $\dtildeV$. 
Since
\begin{equation}
\label{CYDistEq}
\iota_{\calH,\calV}^\star \dV=\dC,
\end{equation}
this will allow us in \S\ref{IntExtKahlerSubsec} to estimate $\dC$ from above in terms
of $\dtildeV$.

\begin{remark}
\label{CYThmRemark}
{\rm
Using only that $(\calH,\gC)$ is isometrically embedded in $(\calV,\gV)$,
the ensuing inequality $\iota_{\calH,\calV}^\star \dV\le\dC$ 
would not suffice to prove our main result.
The isomorphism $\iota_{\calH,\calV}^{-1}:\calV\ra\calH$ 
allows us to compute in $(\calV,\gV)$
and then translate back to $\calH$.
In this sense, the Calabi--Yau isomorphism serves as a change
of variable for the equation
$$
(\Delta_\vp \dot\vp)^2
=
\frac{\o_\vp^n}{\o^n}
\Big(
2\Delta_\vp\ddot\vp-2|\i\ddbar\dot\vp|^2_{\o_\vp}
+\frac{C^2}V\frac{\o_\vp^n}{\o^n}
\Big).
$$
In fact, it seems that being able to compute $\dC$ from (\ref{CYDistEq})
was one of Calabi's original
motivations for his well-known conjecture, and the reason
why his results summarized in \cite{C1} remained unpublished until
recently.
}
\end{remark}

\subsection{Geodesics in Calabi's metric}
\label{GeodesicsCalabiSubsec}

Next, we recall the equation for Calabi's 
geodesics on $\calV$ \cite{C1,Ca}. For completeness,
and in order to fix conventions, we derive the equation 
directly, in a slightly different manner than in \cite{Ca}.
 
Let $\mu\in\calV$.
The energy of a path $\{\mu(t)=F(t)\mu\}\subset\calV$ is given by
$$
\int_{[0,1]\times M}\left(\frac{\mu_t}\mu\right)^2\mu\w dt
=
\int_{[0,1]\times M}
\frac{F_t^2}F\mu\w dt
$$
where $F=F(t,z)$, and subscripts denote differentiation.
Taking the first variation of the energy with respect to
variations $\vp(t,s,z)$ fixing the endpoints, we obtain
that for a geodesic $F(t)$,
$$
0
=\int_{[0,1]\times M}
\frac{2F_{ts}F_tF-F_t^2F_s}{F^2}\mu\w dt
=
\int_{[0,1]\times M}
\big(-2(\log F)_{tt}-((\log F)_t)^2\big)F_s\mu\w dt.
$$
The expression in parantheses is orthogonal to 
$\Omega_0^{2n}(M)$, i.e., constant, and by integrating
against $\mu(t)$ is seen to equal $-\frac1V\Vsq{\mu_t(t)}$.
The equation for geodesics of constant speed $C$ is therefore
\begin{equation}
\label{CalabiGeodEq}
F_t^2-2F_{tt}F-\frac{C^2}VF^2=0,
\end{equation}
where $C^2=\Vsq{\mu_t(t)}$, and when $F>0$ this simplifies 
to $(\sqrt F)_{tt}+\frac {C^2}{4V}\sqrt F=0$.
The unit-speed geodesic connecting $F\mu$ and $G\mu$ thus satisfies
\begin{equation*}
  \sqrt{F(t)} = \frac{ \sin\big(\half(T-t)/\sqrt V\big) }{ \sin\big(\half T/\sqrt V\big) } \sqrt F
  +
  \frac{ \sin\big(\half t/\sqrt V\big) }{ \sin\big(\half T/\sqrt V\big) } \sqrt G.
\end{equation*}
where $T$ is the length of the geodesic. Noting that
\begin{equation}
\label{CalabiGeodInitialVelocityEq}
F_t|_{t=T}\mu=\frac {G\mu}{\sqrt V}\cot\big(\half T/{\sqrt V}\big)
-\frac{\sqrt{FG}\mu}{{\sqrt V}\sin\big(\half T/{\sqrt V}\big)}
\end{equation}
must have unit length yields that 
$T=2\sqrt V\cos^{-1}\big(\frac1V\int\sqrt{FG}\,\mu\big)$.
In fact, geodesics minimize length in $(\calV,\gV)$ 
\cite{C1}, \cite[Lemma 6.3]{Ca}, and so
\begin{equation}
\label{dVDistanceFn}
  \dV(\mu_1, \mu_2) = 2\sqrt V\cos^{-1}\left(
    \frac1V\int_M\sqrt{
        \frac{\mu_1}{\mu_0}
        \frac{\mu_2}{\mu_0}
} 
\, \mu_0
  \right), \end{equation}
where $\mu_0 \in \mathcal{V}$ is any fixed volume form.

\subsection{Intrinsic and extrinsic distance on
the space of fixed-volume volume forms}
\label{IntrinsicExtrinsicVolSubS}
We turn to proving the equivalence of the intrinsic and extrinsic
distance functions on the space of fixed-volume volume forms. The
results of this subsection (excepting Remark \ref{GeodIntersect}) hold
on a general Riemannian (and not necessarily K\"ahler) manifold.

Before stating the result we make some remarks. 
The metric $\gtildeV$ is only a weak Riemannian metric, but it
nevertheless induces a metric space structure \cite[Corollary 11]{Cl2};
we denote the distance function by $\dtildeV$. 
On the other hand, the submanifold $\calV$ is not totally geodesic
in $(\tilv,\gtildeV)$ (nor, equivalently, is its inclusion in $(\calM,\gE)$). 
In fact, no geodesic of the latter is a geodesic of the former; this
follows from the explicit formula for geodesics of $\gtildeV$,
\begin{equation}\label{tilvGeod}
\mu(t)=\Big(1+\frac{t\alpha}{2\mu}\Big)^2\mu, \quad \alpha\in T_\mu\tilv,
\end{equation}
which shows that if $\mu_t(0)=\alpha\in T_{\mu}\calV$ then
$\mu_t(t)=\alpha+\frac t2\big(\frac\alpha\mu\big)^2\mu\not\in
\Omega_0^{2n}(M)$ 
 for any $t\ne0$.
Hence, geodesics in $\tilv$ intersect $\calV$ tangentially in at most
one point, and all other intersections are transverse.
In particular, this implies that $\Vol(M,
\mu(t))$ is not constant.  Additionally, by \eqref{tilvGeod}, 
$\Vol(M, \mu(t))$ is quadratic in $t$.  By
transversality then, $\Vol(M, \mu(t)) = V$ for exactly
one positive value of $t$.  Thus, a geodesic of $\tilv$ that intersects
$\mathcal{V}$ does so in exactly two points.

\begin{remark}
\label{GeodIntersect}
{\rm
We make a slight digression to
observe that, similarly, 
geodesics of $(\mathcal{M}, \gE)$ intersect $\mathcal{H}$
in at most two points.
Indeed, let $\{g(t)\}$ be a geodesic of $(\mathcal{M}, \gE)$ with $g := g(0)$
and $h:= g_t(0) \in T_g \mathcal{H}$.
Denote by $\mu(t) :=  \pi(g(t))$ the volume form induced by $g(t)$,
and $h_0:=h-\frac{1}{2n} \tr(g^{-1} h)$. Then
by \cite[Theorem~2.3]{FG},\cite[Theorem~3.2]{GM},
$    \mu(t) =
    \left(
      \left(
        1 + \frac{t}{4} \tr(g^{-1} h)
      \right)^2 + \frac{n}{8} \tr((g^{-1} h_0)^2) t^2
    \right) \mu(0)$.
  From this, the variation in the total volume of
  $\mu(t)$ is
  \begin{equation*}
    \frac{d}{dt} \int_M \mu(t) = \frac{1}{2} \int_M
    \left(
      \tr(g^{-1} h) + \frac{t}{4} \tr(g^{-1} h)^2 + \frac{n}{2}
      \tr((g^{-1} h_0)^2) t
    \right) \mu(0).
  \end{equation*}
  As $h \in T_g \mathcal{H}$, the first term vanishes, 
  implying that if $\mu(t)$ is tangent to
  $\mathcal{H}$ for some $t \neq 0$, then
  $   \int_M
    \left(
      \tr(g^{-1} h)^2 + 2n \tr((g^{-1} h_0)^2)
    \right) \mu(0) = 0.$
   But this gives that $h = 0$, proving $\{g(t)\}$ intersects $\mathcal{H}$
  tangentially in at most one point. 
Since $\Vol(M, g(t))$ is quadratic in $t$, 
$\{g(t)\}$ intersects
  $\mathcal{H}$ in at most two distinct points (the second point where
  $\Vol(M, g(t)) = V$ might
not be \Kno).

}
\end{remark}

At this point, motivated by \cite{Ca,CG}, making what amounts to a change of coordinates on
$\tilv$ allows for a clearer picture of the geometry of $\calV \subset
\tilv$. 
So fix any $\mu_0 \in \tilv$, and consider the map
$\Phi: \tilv \rightarrow \mathcal{P}$ defined by $\Phi(\mu) := 2
\sqrt{\mu/\mu_0}$, which is seen to be a diffeomorphism.
By \eqref{tilvGeod}, a path $\mu(t) = F(t) \mu_0$ is
a geodesic of $\tilv$ if and only if $(\sqrt{F(t)}\,)_{tt} = 0$, that
is, if and only if $\Phi(\mu(t))_{tt} = 0$.  Thus, in the coordinates
defined by $\Phi$,
$\tilv$ is manifestly flat.
Furthermore, we see that $\Phi$ is an isometry (both in the Riemannian
sense and in the sense of metric spaces) between $(\tilv, \gtildeV)$
and its image in $L^2(M, \mu_0)$, since 
$\gtildeV(\alpha, \alpha) = \norm{d \Phi(\mu) \alpha}_{L^2(M, \mu_0)}^2$.

Also apparent is the fact that $\mathcal{V}$ is a section of a sphere.
Indeed, if $\mu \in \mathcal{V}$, then we have
$\norm{\Phi(\mu)}_{L^2(M, \mu_0)} = 2 \sqrt{V}$, so
$\Phi(\mathcal{V})$ is precisely the intersection of the sphere of
radius $2 \sqrt{V}$ in $L^2(M, \mu_0)$ with $\mathcal{P}$.  Either
from this description, or by using the Cauchy--Schwarz inequality to
see that the argument of $\cos^{-1}$ in \eqref{dVDistanceFn} is
strictly between $0$ and $1$ if $\mu_1 \neq \mu_2$, one sees that
great circles on this spherical section have length strictly less than
$\pi \sqrt{V}$, since $\dV(\mu_1, \mu_2) < \pi \sqrt{V}$ for any
$\mu_1, \mu_2 \in \calV$.  Thus, an arc of a great circle connecting
two boundary points of $\Phi(\mathcal{V})$ is at most a
quarter-circle.

Now, as in \cite{CG}, we can see that geodesics in $\mathcal{V}$ are
projections of chordal geodesics in $\tilv$.  Indeed, let $\mu, \nu
\in \mathcal{V}$ be given, and let $\mu(t)$, for $t \in [0,1]$, be the
unique geodesic of $\tilv$ connecting them.  Then by the above
discussion, $\Phi(\mu(t))$ is simply the line segment (chord) between
$\Phi(\mu)$ and $\Phi(\nu)$.  By elementary geometry, we know that the
geodesic (i.e., arc of a great circle) between $\mu$ and $\nu$ on
$\Phi(\mathcal{V})$ is the projection of $\Phi(\mu(t))$ onto
$\Phi(\mathcal{V})$, which is explicitly given by $\sqrt{V / v(t)}
\Phi(\mu(t))$, where $v(t) = \int_M \mu(t)$.  This arc is
length-minimizing in $L^2(M, \mu_0)$, and since $\Phi$ is an isometry,
its length equals $\dV(\mu, \nu)$.

Using that $\Phi$ is an isometry, we also get the following formula for $\dtildeV$.

\begin{lem}\label{tilvDist}
  Let $\mu_1, \mu_2 \in \tilv$ be given.  Then we have
  \begin{equation*}
    \dtildeV(\mu_1, \mu_2) = \norm{\Phi(\mu_2) - \Phi(\mu_1)}_{L^2(M,
      \mu_0)} = 2
    \left(
      \int_M \left( \sqrt{
            \frac{\mu_2}{\mu_0}
} - \sqrt{
            \frac{\mu_1}{\mu_0}
} \right)^2 \, \mu_0
    \right)^{1/2}.
  \end{equation*}
\end{lem}

\begin{prop}\label{EquivVolForms}
  The intrinsic and extrinsic metrics $\dV$ and
  $\dtildeV$, respectively, are equivalent on $\mathcal{V}$.
  More specifically,
  \begin{equation*}
    \dtildeV \leq \dV <\frac{\pi}{2\sqrt{2}} \dtildeV,
  \end{equation*}
  and these bounds are optimal.
\end{prop}

\begin{proof}
        We first give an essentially algebraic proof of equivalence which
  yields a suboptimal bound.  The optimal bound is given by a
  geometric argument.

  So suppose we are given $\mu_1, \mu_2 \in \mathcal{V}$.
                                                      As noted above, $\dV(\mu_1, \mu_2)<\pi\sqrt V$.
  Hence, by convexity of $x-\frac\pi2\sin x$ on $[0,\pi/2]$,
  \begin{equation*}
    \begin{aligned}
      \frac1{2\sqrt V} \dV(\mu_1, \mu_2) &< \frac{\pi}{2}
      \sin\big(\half \dV(\mu_1, \mu_2)/\sqrt V\big) = \frac{\pi}{2}
      \left( 1 - \cos^2 \big(\half \dV(\mu_1, \mu_2)/\sqrt V\big)
      \right)^{1/2} \\
      &= \frac{\pi}{2} \left( 1 - \left( \frac1V\int_M \sqrt{
                        \frac{\mu_1}{\mu_0}
                        \frac{\mu_2}{\mu_0}
                      } \, \mu_0 \right)^2 \right)^{1/2}.
    \end{aligned}
  \end{equation*}
  Since $1 - x^2 \leq 2 (1 - x)$ for all $x \in \mathbb{R}$, we can
  further estimate
  \begin{equation*}
    \dV(\mu_1, \mu_2) < \pi{\sqrt{2V}} \left( 1 - \frac1V \int_M \sqrt{
                \frac{\mu_1}{\mu_0}
                \frac{\mu_2}{\mu_0}
              } \,
      \mu_0 \right)^{1/2}.
  \end{equation*}
  Since $\int \mu_i=V$, the result follows from Lemma \ref{tilvDist}.

  Now we make use of the geometric discussion preceding the
  proposition.  An arc of a great circle diverges more from a straight
  line the longer it is.  As already noted, great circles between
  boundary points of $\Phi(\mathcal{V})$ can be at most
  quarter-circles.  But such an arc has length equal to 
  $\pi / (2 \sqrt{2})$ times that of the chord between the boundary points.
  Since $\Phi$ is an isometry, we can thus deduce that
  \begin{equation}\label{GeomEquivFactor}
    \dV < \frac{\pi}{2 \sqrt{2}} \dtildeV.
  \end{equation}
  Furthermore, since the bound $\dV(\mu, \nu) < \pi \sqrt{V}$
  given above is optimal, the factor in \eqref{GeomEquivFactor} is
  optimal.
\end{proof}

\subsection{Intrinsic and extrinsic distance on the space
of \K metrics}
\label{IntExtKahlerSubsec}
We are now in a position to prove our first main result.

\begin{theo}
\label{EquivGeneral}
The intrinsic and extrinsic metrics $\dC$ and
  $\dE$, respectively, are equivalent on $\mathcal{H}$.
  More specifically,
  \begin{equation*}
    \frac{1}{\sqrt{2}} \dE \leq \dC < \frac{\pi \sqrt{n}}{4} \dE.
  \end{equation*}
\end{theo}

\begin{proof}
Recall from \S\ref{VolFormSubmanifSubS} that for 
each $g \in \mathcal{M}$, $d\pi|_g$ is an isometry 
between the horizontal tangent space 
$T^h_g\calM\isom T_g\tilv$ at $g$
and the tangent space at $\pi(g)=dV_g\in\tilv$.  In particular, 
if we denote by $\LE$ the $\gE$-length of a path, this implies that
$\LE\big(\{g(t)\}\big) \geq \LE\big(\{\pi(g(t))\}\big)$ for any path $\{g(t)\}$ 
in $\mathcal{M}$, with equality if and only if $g(t)$ is horizontal.
(Note that ``$\LE$'' on the right side of the inequality stands for
the length with respect to $\iota_{\tilv, \mathcal{M}}^\star \gE$.)
In particular, by \S\ref{FixedVolumeFormsSubsec},
\begin{equation}
\label{HorizontalDistEq}
\dtildeV(\mu_1, \mu_2)\leq \sqrt{\frac n2}\dE(\mathcal{M}_{\mu_1},\mathcal{M}_{\mu_2}).
\end{equation}

Let $\omega_1$ and $\omega_2$ denote cohomologous \K forms with volume 
forms $\mu_1$ and $\mu_2$. Combining (\ref{CYDistEq}), Proposition \ref{EquivVolForms},
and  (\ref{HorizontalDistEq}), we have
\begin{equation*}
  \begin{aligned}
    \dC(g_{\omega_1}, g_{\omega_2}) &= \dV(\mu_1, \mu_2) < \frac{\pi}{2\sqrt{2}}
    \dtildeV(\mu_1, \mu_2) \\
    &\leq \frac{\pi \sqrt{n}}{4} \dE(\mathcal{M}_{\mu_1},
    \mathcal{M}_{\mu_2}) \leq \frac{\pi \sqrt{n}}{4} \dE(g_{\omega_1},
    g_{\omega_2}),
  \end{aligned}
\end{equation*}
which is the required upper bound on $\dC$.
This concludes the proof, since the lower bound follows from Proposition \ref{EbinCalabiProp}.
\end{proof}

\section{The completion of $\calH$}
\label{CompletionSection}

In this section, we use the equivalence of $\dtildeV$ and $\dV$ to
first determine the completion of $(\mathcal{V}, \dV)$ 
on a general Riemannian manifold (by ``completion"
we will always mean the metric completion).  From
this, we obtain a simple criterion for the convergence of metrics in
$\mathcal{H}$ with respect to $\dC$. By using recent deep results
from pluripotential theory, this then 
gives a description of the completion of $(\calH,\dC)$.
It can be viewed as giving a geometric description of a 
subset of the class of plurisubharmonic functions
$\calE(M,\o)$ (to be defined below). 

The completions of $\tilv$ and $\mathcal{V}$ can be quickly obtained
using the map $\Phi$ defined in the last section.  
First, though, we recall several elementary facts from functional
analysis.  Let $(X, \mu)$ be a measure space with $\mu(X) < \infty$.

\begin{definition}\label{UnifIntDef}
  A collection $\mathcal{F}$ of measurable functions on $X$ is called
  \emph{uniformly integrable} if, for each $\epsilon > 0$, there
  exists $t \geq 0$ such that for all $f \in \mathcal{F}$,
  \begin{equation*}
    \int_{\{ x \in X \, : \, \lvert f(x) \rvert \geq t \}} \lvert f(x)
    \rvert \, d\mu(x) < \epsilon.
  \end{equation*}
\end{definition}

\begin{lem}[Vitali's Convergence Theorem; {\cite[Thm.~8.5.14]{Ra}, \cite[(13.38)]{HS}}]\label{VitConv}
  A sequence $\{f_k\}$ in $L^p(X, \mu)$ converges to $f \in L^p(X,
  \mu)$ if and only if $f_k$ converges to $f$ in measure and $\{
  \lvert f_k \rvert^p \, : \, k \in \N \}$ is uniformly integrable.
\end{lem}

The next lemma is a simple consequence of Vitali's Convergence
Theorem, but we include its short proof for completeness.

\begin{lem}\label{L1L2}
  A sequence of nonnegative functions $\{ f_k \}$ converges to $f$ in
  $L^2(X, \mu)$ if and only if $f_k^2$ converges to $f^2$ in $L^1(X, \mu)$.
\end{lem}
\begin{proof}
  Suppose $f_k \rightarrow f$ in $L^2(X, \mu)$.  Then by Lemma
  \ref{VitConv}, $f_k \rightarrow f$ in measure, and $\{ \lvert f_k
  \rvert^2 \}$ is uniformly integrable.  Now, consider the sequence
  $f_k^2$.  Clearly $f_k^2 \rightarrow f_k$ in measure.  Also, $\{
  |f_k^2| \}$ is uniformly absolutely continuous, because it is equal
  to the set $\{ |f_k|^2 \}$.
  The converse direction follows in precisely the same way (and is
  also where nonnegativity of the functions is required).
\end{proof}

With these preliminaries, we can determine the completions of $\tilv$
and $\mathcal{V}$.

\begin{theo}\label{ComplCalV}
  Fix $\mu_0 \in \tilv$.  Let $\tilv_0$ denote the space of all
  nonnegative, measurable sections of $\bigwedge^{2n} T^\star M$.
  That is, $\tilv_0$ consists of all tensor fields\ represented in
  local coordinates by 
  $f \,dx^1 \wedge \cdots \wedge dx^{2n}$,   where $f \geq 0$ is a
  measurable, locally defined function; or globally by $F \mu_0$,
  where $F \geq 0$ is a measurable, globally defined function.
    Then the metric completion of $(\tilv, \dtildeV)$ is given by
  \begin{equation*}
    \overline{(\tilv, \dtildeV)} \cong
        \left\{
      \mu \in \tilv_0: \mu / \mu_0 \in L^1(M, \mu_0)
    \right\},
  \end{equation*}
  i.e., the $L^1$ completion of $\tilv$.  (Here, as usual, we identify
  elements that agree up to a $\mu$-nullset.)
  We also have
  \begin{equation*}
    \overline{(\calV, \dV)} \cong 
    \left\{
      \mu \in \overline{(\tilv, \dtildeV)} : \int_M \mu = V
    \right\}.
  \end{equation*}

  Given an element $\mu$ and a sequence $\{ \mu_k \}$ in
  $\overline{(\tilv, \dtildeV)}$ (resp.~$\overline{(\calV, \dV)}$),
  $\{ \mu_k \}$ converges to $\mu$ if and only if $\int_M |\mu -
  \mu_k| \rightarrow 0$.
\end{theo}
\begin{proof}
  We will prove the results of the theorem for $(\tilv, \dtildeV)$; the
  results for $(\mathcal{V}, \dV)$ then follow directly from
  Proposition \ref{EquivVolForms}.

 The completion of $\Phi(\tilv)$ is given by
  \begin{equation*}
    \overline{\Phi(\tilv)}^{L^2(M, \mu_0)} =
    \left\{
      F \in L^2(M, \mu_0) \, : \, F \geq 0\ \textnormal{$\mu_0$-a.e.}
    \right\}
  \end{equation*}
  Thus, the completion of $(\tilv, \dtildeV)$ can be isometrically
  identified with the image of this set under the map $\Phi^{-1}$,
  where we formally extend $\Phi$ and $\Phi^{-1}$ by the same
  algebraic formulas to nonnegative forms and functions, respectively.

            Now, from Lemma \ref{L1L2}, and because $\Phi(\mu)^2 = 4 \mu /
  \mu_0$ for all $\mu \in \tilv$
it follows that
  the
  completion of and convergence in $(\tilv, \dtildeV)$ are those of
  $L^1(M, \mu_0)$, where we identify $2n$-forms with functions via
  $\Phi$.  The statements of the theorem follow.
\end{proof}

Using this, and the isometry between $(\mathcal{V}, \dV)$ and
$(\mathcal{H}, \dC)$, we get the following corollary---a very simple
criterion for convergence with respect to the Calabi metric.

\begin{cor}\label{dCConv}
  A sequence $\{g_k\} \subset \mathcal{H}$ converges to $g \in
  \mathcal{H}$ with respect to $\dC$ if and only if $dV_{g_k}
  \rightarrow dV_g$ in the $L^1$ sense; i.e.,
  \begin{equation*}
    \int_M |dV_g - dV_{g_k}| \rightarrow 0.
  \end{equation*}
\end{cor}

Note that this convergence result improves upon that in the
ambient space $(\mathcal{M}, \dE)$ as given in \cite[Thm.~4.15]{Cl4}.
In that result, it is required in addition that the sequence $\{g_k\}
\subset \mathcal{M}$ converges to $g \in \mathcal{M}$ in measure (when
this is defined in a suitable sense).  
This extra
assumption is essential in the ambient space---where there are, in
contrast to $\mathcal{H}$, many metrics inducing the same volume form.

Since $(\calV,\dV)$ and $(\calH,\dC)$ are isometric metric
spaces, Theorem \ref{ComplCalV} determines the completion of $\calH$
with respect to $\dC$. However, at the moment this is only abstractly,
and not on the level of metrics. To describe $\overline{(\calH,\dC)}$
in terms of metrics, it is necessary to appeal to generalizations
of the Calabi--Yau Theorem that give results about the 
domain and image of the Monge-Amp\`ere operator, 
as we now briefly elaborate.

Let
$$
PSH(M,\o):=\{
\vp\in L^1(M,\o^n)\,:\,
\o+\i\ddbar\vp\ge0, \; \vp \h{\ is upper semi-continuous}
 \}.
$$
For 
$\vp\in PSH(M,\o)\cap C^2(M)$ let $\MA(\vp):=(\o+\i\ddbar\vp)^n$
denote the Monge-Amp\`ere operator.
Much work has gone into understanding what the largest subset of 
$PSH(M,\o)$ is to which MA can be extended in a meaningful way.
Bedford and Taylor were able to define MA on $PSH(M,\o)\cap L^\infty(M)$,
and showed that thus defined, it is continuous under decreasing sequences \cite{BT1,BT2}.
Recently, Guedj--Zeriahi showed that MA can be
further extended to
$$
\calE(M,\o):=\left\{
\vp\in PSH(M,\o)\,:\,
\lim_{j\ra\infty}\int_{\{\vp\le-j\}}(\o+\i\ddbar\max\{\vp,-j\})^n= 0\;
\right\},
$$
maintaining continuity under decreasing sequences in $PSH(M,\o)\cap L^\infty(M)$ 
\cite{GZ2}.
We note that this recent development builds upon the work of many
authors, and we refer to \cite{Di,GZ2} for a historical overview and references.

The class $\calE(M,\o)$ is also important since, by other recent results,
a generalized version of the Calabi--Yau Theorem holds for it. To state these
results we recall that a pluripolar set is by definition a subset $A\subset M$
for which there exists a function $\vp\in PSH(M,\o)$ such that
$A\subset \{\vp=-\infty\}$.
Guedj--Zeriahi proved that if $\mu$ is a nonnegative Borel measure
on $M$ that vanishes on all pluripolar sets, then there exists $\vp\in\calE(M,\o)$
satisfying $\o_\vp^n=\mu$, and Dinew showed that such a $\vp$
is unique up to a constant within $\calE(M,\o)$ \cite{GZ2,Di}.

Returning to our previous discussion, we obtain the following
description of the completion of $(\calH,\dC)$.
Fix a smooth volume form $\mu_0\in\calV$.

\begin{theo}
\label{calHdCCompletion}
The metric completion of $(\calH_\o,\dC)$ is given by
$$
\overline{(\calH_\o,\dC)}\isom
\{\vp\in \calE(M,\o)\,:\, \o^n_\vp/\mu_0\in L^1(M,\mu_0)\},
$$
and is a strict subset of $\calE(M,\o)$.
\end{theo}

\begin{proof}
Let 
$\nu\in \overline{(\calV,\dV)}$
represent an element of the completion.
According to Corollary \ref{dCConv}, $\nu/\mu_0\in L^1(M, \mu_0)$. 
In particular, $\nu$ is absolutely continuous with respect to $\mu_0$, and so
any $\mu_0$-nullset is a $\nu$-nullset. (Here, we regard both $\mu_0$ and
$\nu$ as Borel measures.)  Since locally in $\CC^n$, pluripolar sets
are of Lebesgue measure zero (and hence are contained in a Borel
nullset \cite[11.11(d)]{Ru}), it follows that pluripolar sets are $\mu_0$-nullsets
\cite[\S3.1]{Bl1}.

Thus, by Theorem \ref{ComplCalV} and the aforementioned results of Dinew and Guedj--Zeriahi, 
it follows that we have an isomorphism
$$
\overline{(\calH,\dC)}\isom
\{\o_\vp\in \calE(M,\o)\,:\, \sup\vp=0, \, \o^n_\vp/\mu_0\in L^1(M,\mu_0) \}.
$$
The inclusion 
$$
\overline{(\calH,\dC)}\subset \{\omega_\vp \,:\, \vp\in
\calE(M,\o)
\},
$$ 
is strict,
since measures that charge $\mu_0$-nullsets that are not pluripolar are still
in the image of $\calE(M,\o)$ under MA (such examples exist,
cf. \cite[\S3]{K1},\cite[\S5]{GZ1}). 
\end{proof}

We remark that it would be interesting
to understand the regularity properties of the subclass 
$\overline{(\calH_\o,\dC)}\subset\calE(M,\o)$.

\begin{remark}
{\rm
The simpler geodesic completion of $(\calH,\dC)$ can also be
computed.
From \S\ref{GeodesicsCalabiSubsec}, any unit-speed geodesic emanating
from $g\in\calH$ will satisfy
$$
dV_{g(t)}
= 
dV_{g}
\Big( 
G\sqrt V \sin\big(\half t/\sqrt V\big)+\cos\big(\half t/\sqrt V\big)
\Big)^2
$$
for some $G\in C^\infty(M)$ with $\int_M GdV_{g}=0$ and $\int_MG^2dV_g=1$. 
Thus, the geodesic completion can be identified with metrics whose volume
form is smooth, nonnegative and of mass $V$.  This is because 
$G$ changes sign and so
for some maximal time $T\in(0,\pi\sqrt{V})$, the term in
parentheses above will vanish. It then follows by the work of \Kolodziej\ \cite{K1,K2,K3} 
that there exists a unique $\vp\in PSH(M,\o)\cap L^\infty(M)$ 
that, moreover, is H\"older continuous, such that $\o^n_\vp=n!dV_{g(T)}$.
In general $\vp$ will not be $C^2$ (but see \cite{Bl2} for some additional
regularity statements).

}
\end{remark}

\section{Ricci flow, distance, and stability}
\label{RicciFlowSubsec}

An interesting problem is to understand the relation of the Ricci
flow to the geometry of $(\calM,\dE)$. 
On the other hand, a major problem in \K geometry, often referred to as 
the Yau--Tian--Donaldson conjecture, is to characterize
the existence of K\"ahler-Einstein metrics in terms of
some algebraic or analytic notions of ``stability" (the specification
of the appropriate notion being part of the problem).
Our purpose in this section is twofold.
First, we define an analytic stability notion
and prove that it gives a new characterization
of K\"ahler-Einstein Fano manifolds.
Second, we derive new conditions under which the 
K\"ahler--Ricci flow converges.

\subsection{Calabi--Ricci stability and existence of K\"ahler--Einstein metrics}

In the context of the Yau--Tian--Donaldson conjecture, a
 number of algebraic notions of stability have been introduced,
starting with Tian's notion of K-stability \cite{T}, subsequently refined
by Donaldson \cite{Do1.5} and others.
At present it is still a major open problem to show that such
algebraic notions imply the existence of a K\"ahler--Einstein metric,
although much progress has been made (see, e.g., the surveys \cite{PS,Th}).

The first analytic stability criterion for the 
existence of a K\"ahler-Einstein metric on a Fano manifold was 
obtained by Tian \cite{T} in terms of the properness 
of the Mabuchi K-energy \cite{M0}, and this has later been
extended to other energy functionals \cite{SW,R0}.
Another, conjectural, notion of stability is that of ``geodesic
stability," due to Donaldson \cite{Do1} (see also Chen \cite{Ch2}).

\begin{conj} 
\label{DonaldsonConj}
{\rm (See \cite{Do1})} The following are equivalent:\hfill\break
(i) There exists no constant scalar curvature metric in $\calH$.\hfill\break
(ii) There exists a geodesic ray in $(\calH_\o,\gM)$ along which 
the derivative of the K-energy is negative.\hfill\break
(iii) There exists a geodesic ray as in (ii) starting at any point in $\calH$.
\end{conj}

For the remainder of this section, we assume
$(M,J,\o)$ is
a Fano manifold (i.e., its first 
Chern class $c_1(M)$ is positive), with $[\o]=c_1(M)$.
In stark contrast to the geodesic flow of $\calM$, which 
instantly leaves $\calH$ (Remark \ref{GeodIntersect}),
Hamilton's Ricci flow on $\calM$ (which we will always assume is volume normalized),
defined by 
\begin{equation}
\label{RFEq}
\frac{\del \o(t)}{\del t}=-\Ric \o(t)+\o(t),\quad \o(0)=\o\in\calH,
\end{equation}
preserves $\calH$ \cite{H}
and exists for all $t>0$ by a theorem of Cao \cite{Cao}.

We introduce the following notion of analytic stability.

\begin{definition}
\label{CRDef}
We say that $(M,J)$ is Calabi--Ricci unstable (or CR-unstable)
if there exists a Ricci flow 
that diverges
in $\overline{(\calH,\dC)}$.
Otherwise, we say $(M,J)$ is CR-stable.
\end{definition}

Note that the derivative of the K-energy is negative
along the Ricci flow (\ref{RFEq}). Also, as will follow from the results below,
the definition of CR-instability can be equivalently phrased
in terms of the existence of a Ricci flow of infinite $\dC$-length.
Thus, this definition stands in precise analogy to Donaldson's
geodesic stability, with Ricci flow paths and $\dC$-distance taking the place of 
$\gM$-geodesic rays and $\dM$-distance.

Definition \ref{CRDef} is motivated by the following result, that
similarly stands in clear analogy to Donaldson's conjecture.

\begin{theo}
\label{CRStabilityThm}
A Fano manifold $(M,J)$ is CR-stable if and only if it admits a \KE metric.
Moreover, if it is CR-unstable then any Ricci flow 
diverges in $(\calH,\dC)$.
\end{theo}

\begin{proof}

First, assume that a K\"ahler-Einstein metric exists.
Then, a theorem of Perelman and work of Chen--Tian, Tian--Zhu and
Phong--Song--Sturm--Weinkove implies that any Ricci
flow (\ref{RFEq}) will converge to a \KE metric exponentially fast in any $C^k$ norm
\cite{ChT0,TZ,PSSW}.  In particular, the metrics along the flow are
uniformly equivalent, and  
$\big|\frac{\del\o}{\del t}\big|_{\o(t)}
<C \big|\frac{\del\o}{\del t}\big|_{\o}
<C_1e^{-C_2t}$,
for some $C_1,C_2>0$ independent of $t$.
Hence the $\dC$-length of $\{\o(t)\}_{t\ge0}$ is finite and the
flow converges in $\dC$.

Now, assume that $(M,J)$ admits no K\"ahler-Einstein metric.
The flow (\ref{RFEq}) induces the K\"ahler--Ricci flow
\begin{equation}
\label{KRFEq}
\o_{\vp(t)}^n=\o^n e^{f_\o-\vp(t)+\dot\vp(t)},\quad \vp(0)=c_0,
\end{equation}
on $\calH_\o$, where $\i\ddbar f_\o=\Ric\o-\o$ and $\frac1V\inf_Me^{f_\o}\o^n=1$.
The initial condition $c_0$ is a certain constant uniquely determined
by $\o$, fixed once and for all \cite{ChT0,PSS}.

Recall that the multiplier ideal sheaf associated to a
function $\vp\in PSH(M,\o)$ is defined as the sheaf $\calI(\vp)$
defined for each open set $U\subset M$ by local sections
$$
\calI(\vp)(U)=\{h\in\calO_M(U)\,:\,|h|^2e^{-\vp}\in L^1_{\h{\smlsev loc}}(M)\}.
$$
Such a sheaf is called proper if it is neither zero nor the structure sheaf $\calO_M$,
and is called a Nadel sheaf whenever there 
exists $\epsilon>0$ such that $(1+\epsilon)\vp\in PSH(M,\o)$.
We recall the following result, describing the limiting behavior of
the K\"ahler--Ricci flow in terms of a Nadel multiplier ideal sheaf.

\begin{theo} {\rm \cite[Theorem 1.3]{R3}}
\label{NadelRFThm}
Let $(M,J)$ be a Fano manifold not admitting a \KE metric. Let 
$\gamma\in(n/(n+1),1)$ and let $\o\in\calH$. Then there 
exists 
a subsequence $\{\vp(t_j)\}_{j\ge1}$ of solutions
of (\ref{KRFEq}) with $\lim_{j\ra\infty} t_j=\infty$,
such that $\vp(t_j)-\frac1V\int_M\vp(t_j)\o^n$
converges in the $L^1(M,\o)$-topology to $\vp_\infty\in PSH(M,\o)$
and $\calI(\gamma\vp_\infty)$ is a proper Nadel multiplier ideal
sheaf.

\end{theo}

Now, fix some $\gamma\in(n/(n+1),1)$, and let $\calI(\gamma\vp_\infty)$
be the Nadel sheaf constructed by Theorem \ref{NadelRFThm}.
This sheaf cuts out a subscheme in $M$ whose support, 
which we denote by $S$, is a nonempty subvariety of 
positive codimension \cite{N1}.

The following lemma is an analogue for the Ricci flow 
of a well-known fact for the continuity method \cite{T0,N2}.

\begin{lem}
\label{VolumeConcLemma}
Let $K\subset M\setminus S$ be a compact set, and let 
$\{\vp(t_j)\}_{j\ge1}$ be as in Theorem \ref{NadelRFThm}. Then
$$
\lim_{j\ra\infty}\int_K\o^n_{\vp(t_j)}=0.
$$
\end{lem}

\begin{proof}
Given the Sobolev inequality along the Ricci flow \cite{Ye,Z} and Perelman's 
deep estimates for the K\"ahler--Ricci flow \cite{ST,TZ}, 
the proof follows in the same way as the corresponding
result for the continuity method \cite[Prop. 4.1]{N2},  
and so we only outline the proof for completeness. 

By (\ref{KRFEq}) and Perelman's estimate 
(see \cite{ST,ChT0,PSS},
or \cite[Theorem 2.1(i)]{R3})
\begin{equation}
\label{PerelmanEstEq}
|\dot\vp(t)|<C,  
\end{equation}
it suffices to estimate $\int_Ke^{-\vp(t_j)}\o^n$.
Fix $\gamma\in(n/(n+1),1)$. From \cite[(16), (21)]{R3}
\begin{equation}
\label{CxSingExpEq}
\lim_{j\ra\infty}\int_M e^{-\gamma(\vp(t_j)-\sup\vp(t_j))}=\infty.
\end{equation}
It follows then from \cite[Theorem 3.1]{T0} that 
\begin{equation}
\label{KEstEq}
\int_K e^{-\gamma(\vp(t_j)-\sup\vp(t_j))}\o^n<C
\end{equation}
for some uniform constant $C$ depending on $K$ (here and in the statement it might
be necessary to take a subsequence, but we omit this from the notation).
By (\ref{KEstEq}), $$
\int_Ke^{-\vp(t_j)}\o^n\le Ce^{-\gamma\sup\vp(t_j)-(1-\gamma)\inf\vp(t_j)}.
$$
To conclude it suffices to use the Harnack inequality 
$-\inf\vp(t)\le n\sup\vp(t)+C$. Note that in \cite[(15)]{R3} 
the Harnack inequality 
$-\frac1V\int_M\vp(t)\o^n_{\vp(t)}\le n\sup\vp(t)+C$
is proved, and that as in \cite{T0,T0.5} (cf. \cite{Siu}) one can then
deduce from it the previous inequality. Alternatively,
the former follows from the latter via a Green's function estimate 
(cf. \cite[p.~626]{PSS}, \cite[p.~5847]{R3}). Indeed, 
$-\inf\vp(t)\le-\frac1V\int_M\vp(t)\o^n_{\vp(t)}+nA_t$,
where $-A_t$ is the minimum of the Green function of
$(M,\o(t))$, normalized to have average zero (see \cite[p.~5845]{R3}).
As shown by Bando--Mabuchi \cite[(3.4)]{BM}, the 
heat kernel estimate of Cheng--Li \cite[(2.9)]{CL} implies
such an estimate as soon as one has uniform
Poincar\'e and Sobolev inequalities
(assume $n>1$, as $n=1$  is treated in \cite[p.~5847]{R3}), 
and these indeed
hold, by the results of Perelman, Ye and Zhang \cite[Theorem 2.1]{R3}.
Here we note (as pointed out to us by V. Tosatti) 
that the classical ``weighted'' Poincar\'e inequality
\cite[Lemma 2.3]{R3} implies the usual Poincar\'e inequality
by a straightforward argument using the Cauchy-Schwarz inequality.
\end{proof}

It follows from this lemma that $\vp_\infty\not\in\calE(M,\o)$, and
hence by Theorem \ref{calHdCCompletion} that the
limit point $\o_{\vp_\infty}$ of the Ricci flow  
is not in $\overline{(\calH,\dC)}$.
Thus, the Ricci flow with initial condition $\o(0)=\o$
does not converge with respect to $\dC$, and
so $(M,J)$ is CR-unstable.
Since $\o\in\calH$ was arbitrary, this concludes the proof.
\end{proof}

\subsection{Analytic criteria for the convergence 
of the K\"ahler--Ricci flow}

The goal of this subsection is to derive new analytic
criteria for the convergence of the K\"ahler--Ricci flow.

First, we show that $\dC$-convergence of the flow
implies an a priori $C^0$ estimate. Such an estimate
does not follow from \Kolodziej's deep results \cite{K1,K2}
that require slightly stronger control on the volume form than $L^1$.

\begin{theo}
\label{dCCzeroThm}
Assume that the K\"ahler--Ricci flow (\ref{KRFEq}) $\dC$-converges,
i.e., assume that $\{\o^n_{\vp(t)}\}_{t\ge0}$ converges in $L^1$
(see Corollary \ref{dCConv}).
Then there exists a constant $C>0$ independent of $t$ such that
$||\vp(t)||_{C^0(M)}<C$.
\end{theo}

\begin{proof}
The proof can be extracted from the proofs of
Theorem \ref{CRStabilityThm} and \cite[Theorem 1.3]{R3},
but we summarize it below for the reader's convenience.

By \cite[(15), (16), (20)]{R3}, the following a priori estimates
hold:
$$
\begin{aligned}
\frac1V\int_M-\vp(t)\o^n_{\vp(t)}
&\le \frac nV\int_M\vp(t)\o^n,
\cr
\sup\vp(t)&\le \frac1V\int_M\vp(t)\o^n +C,
\cr
-\inf\vp(t)&\le \frac CV\int_M-\vp(t)\o^n_{\vp(t)}.
\end{aligned}
$$
Let $\gamma\in(n/(n+1),1)$.
These inequalities, combined with Perelman's
estimate (\ref{PerelmanEstEq}), imply
that a uniform bound
on $\int_M e^{-\gamma(\vp(t)-\frac1V\int_M\vp(t)\o^n)}\o^n$
leads to a uniform estimate $||\vp(t)||_{C^0(M)}<C$.

Hence, supposing that $\{||\vp(t)||_{C^0(M)}\}_{t\in[0,\infty)}$ is unbounded,
it follows by \cite{R3} that one can find a subsequence 
$\{\vp(t_j)\}_{j\ge1}$
as in Theorem \ref{NadelRFThm}, which by Lemma \ref{VolumeConcLemma}
satisfies 
$\lim_{j\ra\infty}(\vp(t_j)-\frac1V\int_M\vp(t_j)\o^n)=\vp_\infty
\in PSH(M,\o)\setminus\calE(M,\o)$. 
Thus, one can conclude as in the proof of Theorem \ref{CRStabilityThm}.
\end{proof}

As is well-known, once the crucial $C^0$ estimate is
established for the flow, higher derivative estimates then
follow \cite{Y,Cao,ChT0,TZ,PSS,Pa}, and one obtains
smooth convergence up to automorphisms \cite{ChT0,PSS,TZ}.
On the other hand, a theorem of Phong et al. \cite{PSSW} shows 
that the latter convergence implies exponential convergence
of the original flow.
To summarize, we have the following statement that
shows that 
the very weak notion of $\dC$-convergence
(Corollary \ref{dCConv}) implies such strong convergence.

\begin{cor}
\label
{dCConvSmoothConvCor}
If $\{\o^n_{\vp(t)}\}_{t\ge0}$ converges in $L^1$, i.e., 
if the K\"ahler--Ricci flow (\ref{RFEq}) $\dC$-converges,
then it converges smoothly (exponentially fast).
\end{cor}

In the ambient space $\mathcal{M}$, 
it seems to be
difficult to find 
(nontrivial) 
settings under which
$\dE$-convergence implies a more synthetic-geometric notion of convergence, such
as Cheeger--Gromov, Lipschitz, or even Gromov--Hausdorff. Indeed, examples
show that $\dE$-convergence is too weak to control the geometry in any
way (for
further discussion, see \cite[\S 4.3]{Cl4}, \cite[\S 1, \S
5]{Cl5} and cf. \cite{An}), and one can use 
Theorem \ref{EquivGeneral} and
Corollary \ref{dCConv} to construct examples
that show this is also the case for $\calH\subset\calM$.
The previous corollary thus provides a rather striking instance of
such a setting.

Let $s(t):=\tr_{\o(t)}\Ric\o(t)$ denote the scalar curvature along the flow.
We also record the following weaker version of Corollary \ref{dCConvSmoothConvCor}.

\begin{cor}
\label
{LoneLtwoConvCor}
The K\"ahler--Ricci flow (\ref{RFEq}) converges smoothly if and only if
\begin{equation}
\label{LoneLtwoConvEq}
||s-n||_{L^1(\RR_+,L^2(M,\o(t)))}<\infty,
\end{equation}
i.e., if and only if it has finite $\dC$-length.

\end{cor}

This improves a result of Phong et al.~\cite{PSSW}, where
$||s-n||_{L^1(\RR_+,C^0(M))}<\infty$ is assumed instead.

\begin{remark}
{\rm

The problem of finding conditions for the convergence
of the K\"ahler--Ricci flow has been studied by many authors.
Many of these results involve assuming that the  
curvature tensor is uniformly bounded along the flow, combined
with some further analytic and algebraic conditions:
Phong et al.~assumed uniform curvature bounds, the vanishing
of the Futaki invariant and a certain stability condition
on the complex structure \cite{PS0,PSSW0}; 
Sz\'ekelyhidi \cite{Sz} assumed that the K-energy is bounded from below 
and the manifold is K-polystable (together with the curvature
bounds), and Tosatti \cite{To} then replaced the K-energy bound
by assuming asymptotic Chow semistability. 
Other results include:
convergence when the K-energy is bounded below and 
the first eigenvalue of the Laplacian on $T^{1,0}M$ 
is uniformly positive \cite{PSSW} (cf.~\cite{MSz,ZhZ}), and 
convergence when the evolving volume forms
satisfy $\o(t)^n\ge C\o^n$ \cite{Pa}. The list above is by no means
exhaustive and we refer to these articles for further references.

}
\end{remark}

\section{Some remarks and 
further study}
\label{FutureStudySection}

We end with some remarks and indicate some possible
directions for future study.

\subsection{Angles between $\mathcal{H}$ and conformal classes}
\label{sec:angl-betw-mathc}

Lemma \ref{TanHTrace}
is, geometrically, a statement that $\mathcal{H}$ intersects conformal
classes in $\mathcal{M}$ at a constant angle.  For $g \in
\mathcal{H}$, we consider the conformal class $\mathcal{P}g$, where we
recall that $\mathcal{P}$ denotes the group of smooth positive
functions on $M$, acting on $\mathcal{M}$ by pointwise multiplication.
Tangent vectors to $\mathcal{P}g$ are of the form $\rho g$ for $\rho
\in C^\infty(M)$.

Let $h \in T_g \mathcal{H}$ and $k\in T_g(\calP g)$ 
with $\Enorm{h} = \Enorm k=1$ be given.
Denote the pure trace part of $h$ by $h_T := \frac{1}{2n} \tr(g^{-1}
h) g$, and the traceless part of $h$ by $h_0 := h - h_T$.  On the one
hand, the decomposition $h = h_0 + h_T$ is orthogonal, so we have
$\E hh=:\Esq{h} = \Esq{h_0} + \Esq{h_T}$.  
On the other hand, by Lemma \ref{TanHTrace},
$\Esq{h} = 
 n \Esq{h_T},$ 
or $\Enorm{h_T} = n^{-1/2}$.
We then have 
$\E hk = \E {h_T}k \leq \Enorm{h_T}\Enorm k = n^{-1/2},$
with equality if and only if $k = \sqrt{n} h_T$.
Hence, the angle between $\mathcal{H}$ and
$\mathcal{P} g$ is $\cos^{-1}(n^{-1/2})$, independently of $g \in
\mathcal{H}$.  In particular, in the case $n = 1$, the angle is
$0$---reflecting the fact that $\mathcal{H}$ is contained within a
single conformal class.  For $n = 2$, the angle is $\pi/4$, and
the tangent space to the space of \K metrics lies ``halfway" between those
of $\calP g$ and $\calM_\mu$, and as $n$ grows
the solution of the Calabi--Yau equation diverges more and 
more from a conformal transformation.

Note also that given the geometric description of $\mathcal{V} \subset
\tilv$ in 
\S\ref{IntrinsicExtrinsicVolSubS}, the space
$(\mathcal{V}, \dV)$ can be readily shown to have constant positive
curvature, both in the sense of the Riemannian sectional curvature,
and well as in the synthetic sense of Alexandrov.

To summarize, $\mathcal{H}$ is a section of a sphere isometrically
embedded in $\mathcal{M}$, and each conformal class $\mathcal{P} g$ is
an incomplete, isometrically embedded Euclidean domain.  Since each of
these flat spaces intersects the spherical section at the same angle,
we expect the former to all meet in the completion of $(\calM,\dE)$---as 
indeed they do, 
at the point represented by the zero tensor (this follows from
\cite[Prop.~4.1]{Cl3}, which implies that $\dE(\lambda g_1, \lambda g_2)
\rightarrow 0$ as $\lambda \rightarrow 0$ for any $g_1, g_2 \in
\mathcal{M}$).

\subsection{Other metrics on $\calH$ and on $\calM$} 
\label
{OtherMetricsSubsection}

As mentioned in the introduction,
one may consider different metrics on $\calH_\o$. 
Currently of greatest interest, perhaps, is the Mabuchi
metric $\gM$ (\ref{MabuchiMetricEq}), in part due to its intimate relation
to several important problems in \K geometry concerning the existence of
canonical metrics \cite{M,Do1,Ch1,Ch2,ChT}.
Geodesics of $\gM$ are solutions of a 
homogeneous complex Monge-Amp\`ere equation (HCMA). At present it is not
known how to construct the exponential map of $\gM$, or equivalently
how to solve the Cauchy problem for the HCMA. This problem seems
quite difficult due to issues of ill-posedness, and it seems plausible
that most directions will not exponentiate to geodesics in $\calH$.
We refer to \cite{RZ1,RZ2} for more precise statements.
Even the more standard Dirichlet problem of constructing a geodesic
between two given metrics is not completely understood, although
much progress has been made by Chen and Tian toward a partial 
regularity theory \cite{Ch1,ChT}, and this has been used
by Chen to study the geodesic distance induced by $\gM$ \cite{Ch2}.

Thus, while geodesics in $\gC$ are not completely
explicit, they are still considerably simpler to understand, 
and it would be of interest to compare the Mabuchi geometry 
to that of Calabi. Since the length of
smooth minimizing geodesics in $(\calH_\o,\gM)$ need not be uniformly bounded (e.g.,
a one-parameter family of automorphisms induces a geodesic line),
one may only ask whether the Mabuchi distance dominates the Calabi
distance. An analogous problem would be to compare the Donaldson metric \cite{Do3}
(see also \cite{ChH}), which is an analogue of the Mabuchi metric on $\calV$, 
to the metric $\gV$ induced from $(\calM,\gE)$. 
Another metric on $\calH$ can be defined by the $L^2$ norm of
the gradient \cite{Ca}, 
and again, it would be interesting to
compare it to $\gC$ and $\gM$. 
Similarly, one could also consider metrics involving more
derivatives than $\dC$, and these can be induced by
metrics on $\calM$. Stronger metrics might lead to
notions of convergence that could be of more use in various
geometric settings, where the rather weak notion of convergence
associated with $\gE$ is often insufficient (cf. \cite{Cl4}).

\subsection{Other submanifolds of metrics}
As noted in \S\ref{InducedMetricSection}, the Calabi metric, intrinsically
defined in terms of \K potentials, is obtained by restriction
of the $L^2$ metric on $\calM$. One possible application 
of this is that one can define natural metrics
on other submanifolds of $\calM$,
e.g., spaces of almost-\K metrics, where the $\ddbar$-lemma is absent.
The question of whether the induced geometry can be understood
successfully is then essentially equivalent to whether a version 
of the Calabi--Yau Theorem exists in those settings, itself a topic
of current research \cite{Do2,TWY}.

\subsection{K\"ahlerian and Riemannian extensions of CR-stability}
\label{CRExtensionSubsec}

Definition \ref{CRDef} can be extended to an arbitrary polarized
\K manifold, for instance, by considering generalizations of the K\"ahler-Ricci
flow whose stationary points are constant scalar curvature 
or extremal metrics (see \cite[\S3]{R3} and \cite[\S7-9]{R2}). 
A natural question is whether a result corresponding to Theorem \ref{CRStabilityThm}  
holds in these more general cases.
We also remark that multiplier ideal sheaves can also be constructed 
for the Ricci iteration, a discrete version of the K\"ahler-Ricci flow introduced in \cite{R2},
for which a similar, and likely equivalent,  notion of stability may be defined.

On the other hand,
by the equivalence of $\dC$ and $\dE$ (Theorem \ref{EquivGeneral})
and the fact that the (volume normalized) Ricci flow preserves $\calH$,
the notion of CR-stability for a Fano manifold
is a purely Riemannian one, i.e., it can be stated in terms
of $(\calM,\dE)$.
Thus this notion
can be extended naturally to any Riemannian
manifold.
It is then an interesting problem whether Theorem \ref{CRStabilityThm}
has a suitable analogue for the Ricci flow and Einstein metrics
in this more general setting.

\bigskip
\noindent
{\bf Acknowledgements.}
This material is based upon work supported in part by NSF grants 
DMS-0902674, 0802923.
We thank S.~Dinew and V. Tosatti for useful discussions and B.~Klartag for
telling us about \cite{CG}.

\end{document}